\documentclass[12pt,a4paper]{article}

\usepackage{theorem,enumerate}
\usepackage{amsmath,latexsym,amssymb,amsfonts}
\usepackage{eucal}
\usepackage{eufrak}
\usepackage{mathrsfs}
\usepackage{ulem}

\newcounter{Scounter}
\theorembodyfont{\normalfont\slshape}
\newtheorem{thm}{Theorem}

\newtheorem{Thm}{Theorem}

\newtheorem{lem}[thm]{Lemma}

\newtheorem{Q}[thm]{Question} 

\newtheorem{claim}{Claim}[section] 
\newtheorem{subclaim}{Subclaim}[claim]

\newcommand{\proof}{\medbreak\noindent\textit{Proof.}\quad}

\newcommand{\qed}{{$\quad\square$\vs{3.6}}}

\newcommand{\vs}[1]{\vspace*{#1 mm}}

\numberwithin{equation}{section}

\bfseries\normalfont

\title{On the number of $4$-contractible edges in plane triangulations}

\author{
Toshiki Abe$^{1}$
\and
Michitaka Furuya$^{2}$
\and
Raiji Mukae$^{3}$
\and
~~~~~~~~~Shoichi Tsuchiya$^{4}$\footnote{Email address: \texttt{s.tsuchiya@isc.senshu-u.ac.jp}}\and
\\
\small
$^{1}$\small\textsl{National Institute of Technology, Miyakonojo College}\\
$^{2}$\small\textsl{School of Engineering, Kwansei Gakuin University}\\ 
$^{3}$\small\textsl{Faculty of Education, University of Miyazaki}\\ 
$^{4}$\small\textsl{School of Network and Information, Senshu University}\\ 
}

\date{}

\begin{document}

\maketitle

\begin{abstract}
In 2007, Ando and Egawa proved a theorem
which provides a lower bound on the number of contractible edges preserving $4$-connectedness in $4$-connected graphs.
In this paper, we refine their bounds, especially for the $4$-connected plane triangulations.
In particular, we show that
if $G$ is a $4$-connected plane triangulation of order at least $7$, then $G$ contains at least $|V_{\ge 5}|+2$ contractible edges preserving $4$-connectedness, where $V_{\ge 5}$ is the set of vertices of degree at least $5$.
We also determine the extremal graphs. 
\end{abstract}



\section{Introduction}\label{sec1}

In this paper, we only deal with finite graphs without loops and multiple edges.
For a graph $G$, $|G|$ denotes the number of vertices of $G$.
For $u\in V(G)$, let $N_{G}(u)$ and $d_{G}(u)$ denote the {\sl neighborhood} and the {\sl degree of $u$}, respectively; thus $N_{G}(u)=\{v\in V(G):uv\in E(G)\}$ and $d_{G}(u)=|N_{G}(u)|$.
Also let $N_{G}[u]$ denote the {\sl closed neighborhood} of $u$, 
i.e., we define $N_{G}[u] = N_{G}(u) \cup \{ u \}$.
Similarly, for $X \subset V(G)$, let $N_{G}(X)$ and $N_{G}[X]$ denote the {\sl neighborhood} and {\sl closed neighborhood} of $X$, i.e., we define $N_{G}(X) =\{v\in V(G)-X: u \in X, uv\in E(G)\}$ and $N_{G}[X] = N_{G}(X) \cup X $.
For a graph $G$ and $X \subset V(G)$, let $G[X]$ denote the subgraph of $G$ induced by $X$.
A graph $G$ of order at least $k-1$ is said to be {\sl $k$-connected} if for any vertex sets $S \subset V(G)$ with $|S| \le k-1$, $G-S$ is connected.
For a $k$-connected graph $G$, an edge $e$ is called {\sl $k$-contractible} 
if the graph obtained from $G$ by contracting $e$ preserves $k$-connectedness.
Let $E_k$ be the set of $k$-contractible edges in $G$.
Tutte~\cite{T} proved that if $G$ is a $3$-connected graph, then $E_3 \neq \emptyset$.
Later, Ando, Enomoto, and Saito \cite{AES} proved a theorem guaranteeing the existence of many $3$-contractible edges in $3$-connected graphs. 
Also, they characterized its extremal graphs.
On the other hand, Martinov~\cite{M} characterized $4$-connected graphs $G$ with $E_4 = \emptyset$.
According to this result, we cannot guarantee $4$-contractible edges with linear order of $|G|$ in $4$-connected graphs. 
Ando, Egawa, Kawarabayashi, and Kriesell gave a bound on $|E_4|$ as follows.

\begin{Thm}[Ando, Egawa, Kawarabayashi and Kriesell \cite{AEKK}]\label{thm:AEKK}
Let $G$ be a $4$-connected graph.
Then we have $|E_4| \ge \frac{1}{34} \cdot (|E(G)|-2|G|)$.
\end{Thm}

In \cite{EN}, another bound of $|E_4|$ was investigated.
In a graph $G$, let $V_k$ denote the vertex set of $G$ consisting of all vertices of degree $k$.
Similarly, let $V_{\ge k}$ denote the vertex set of $G$ consisting of all vertices of degree at least $k$.
Also, Ando and Egawa proved the following.

\begin{Thm}[Ando and Egawa \cite{AE}]\label{thm:AE}
Let $G$ be a $4$-connected graph.
Then we have $|E_4| \ge |V_{\ge 5}|$.
\end{Thm}

A {\sl plane graph} is a graph embedded in the plane without edge-crossings.
If all faces of a plane graph $G$ are triangular, then $G$ is called a {\sl plane triangulation}.  
We note that we do not identify $K_3$ as a plane triangulation.
We can see that every triangulation on a surface is $3$-connected.
Note that Theorems \ref{thm:AEKK} and \ref{thm:AE} do not guarantee 
the number of contractible edges in triangulations on the surfaces other than the plane 
(for the details, see \cite{AFMT}).

Remark that $K_4$ is the unique $3$-connected graph with order $4$, where $K_n$ is a complete graph of order $n$.
In~\cite{AFMT}, it was proved that for a plane triangulation $G$ of order at least $5$, we have $|E_3| \ge |G|+\frac{1}{2}|V_3|$. 

By the Euler's formula, if $G$ is a plane triangulation, then we have $|E(G)|=3|G|-6$, 
and hence we have $|E_4| \ge \frac{1}{34} \cdot (|G|-6)$ by Theorem~\ref{thm:AEKK}.
In a sense, this bound is best possible for $4$-connected plane triangulations 
because the octahedron (i.e., the graph isomorphic to a double wheel of order $6$) 
has no $4$-contractible edges.
However, if we assume $|G|$ is sufficiently large, 
then we can guarantee the existence of more $4$-contractible edges in a $4$-connected plane triangulation.
For a $4$-connected plane triangulation, McCuaig, Haglin, and Venkatesan proved the following.

\begin{Thm}[McCuaig, Haglin and Venkatesan \cite{MHV}]\label{thm:MHV}
Let $G$ be a $4$-connected plane triangulation of order at least $8$.
Then we have $|E_4| \ge \frac{3}{4}|G|$. 
\end{Thm}

It was proved that there are infinitely many $4$-connected graphs which attain the equality in Theorems~\ref{thm:AE} or \ref{thm:MHV}.
In particular, Ando and Egawa~\cite{AE} constructed an infinite family of $4$-connected plane graphs
with $|E_4| = |V_{\ge 5}|$.
In this paper, we prove that 
if $G$ is a $4$-connected plane triangulation of order at least $7$, 
then we have $|E_4| \ge |V_{\ge 5}|+2$.
Also, we characterize extremal graphs.

Let $G_0$ be the plane triangulation as depicted in Figure~\ref{fig_2}.
Remark that for this triangulation, we have $|E_4|=|V_{\ge 5}|+2$ 
because it contains four vertices in $V_{\ge 5}$ and six edges in $E_4$.

\begin{figure}[htb]
\begin{center}
{\unitlength 0.1in%
\begin{picture}(20.1000,18.4100)(11.9800,-36.3100)%
%
\special{pn 8}%
\special{pa 1591 2199}%
\special{pa 2391 2999}%
\special{fp}%
\special{pa 2391 2999}%
\special{pa 2391 2999}%
\special{fp}%
%
\special{pn 8}%
\special{pa 2391 2999}%
\special{pa 2391 2199}%
\special{fp}%
\special{pa 2391 2199}%
\special{pa 1591 2199}%
\special{fp}%
\special{pa 1591 2199}%
\special{pa 1591 2999}%
\special{fp}%
\special{pa 1591 2999}%
\special{pa 2391 2999}%
\special{fp}%
%
\special{sh 1.000}%
\special{ia 1581 2209 41 41 0.0000000 6.2831853}%
\special{pn 8}%
\special{ar 1581 2209 41 41 0.0000000 6.2831853}%
%
\special{sh 1.000}%
\special{ia 2381 3009 41 41 0.0000000 6.2831853}%
\special{pn 8}%
\special{ar 2381 3009 41 41 0.0000000 6.2831853}%
%
\special{sh 1.000}%
\special{ia 1581 3009 41 41 0.0000000 6.2831853}%
\special{pn 8}%
\special{ar 1581 3009 41 41 0.0000000 6.2831853}%
%
\special{sh 1.000}%
\special{ia 2381 2209 41 41 0.0000000 6.2831853}%
\special{pn 8}%
\special{ar 2381 2209 41 41 0.0000000 6.2831853}%
%
\special{pn 8}%
\special{pa 2391 2199}%
\special{pa 1791 2399}%
\special{fp}%
\special{pa 1791 2399}%
\special{pa 1591 2999}%
\special{fp}%
\special{pa 1591 2999}%
\special{pa 2191 2799}%
\special{fp}%
\special{pa 2191 2799}%
\special{pa 2391 2199}%
\special{fp}%
%
\special{sh 0}%
\special{ia 2181 2809 41 41 0.0000000 6.2831853}%
\special{pn 8}%
\special{ar 2181 2809 41 41 0.0000000 6.2831853}%
%
\special{sh 0}%
\special{ia 1781 2409 41 41 0.0000000 6.2831853}%
\special{pn 8}%
\special{ar 1781 2409 41 41 0.0000000 6.2831853}%
%
\special{pn 8}%
\special{pa 2391 2199}%
\special{pa 2417 2227}%
\special{pa 2442 2256}%
\special{pa 2468 2284}%
\special{pa 2518 2340}%
\special{pa 2542 2369}%
\special{pa 2566 2397}%
\special{pa 2589 2425}%
\special{pa 2612 2454}%
\special{pa 2654 2510}%
\special{pa 2674 2539}%
\special{pa 2693 2567}%
\special{pa 2711 2596}%
\special{pa 2727 2624}%
\special{pa 2742 2653}%
\special{pa 2756 2682}%
\special{pa 2769 2710}%
\special{pa 2779 2739}%
\special{pa 2788 2768}%
\special{pa 2796 2797}%
\special{pa 2801 2826}%
\special{pa 2805 2855}%
\special{pa 2806 2884}%
\special{pa 2806 2913}%
\special{pa 2803 2942}%
\special{pa 2798 2971}%
\special{pa 2791 3000}%
\special{pa 2781 3030}%
\special{pa 2769 3059}%
\special{pa 2755 3088}%
\special{pa 2738 3117}%
\special{pa 2720 3145}%
\special{pa 2700 3173}%
\special{pa 2679 3200}%
\special{pa 2656 3225}%
\special{pa 2632 3250}%
\special{pa 2607 3274}%
\special{pa 2580 3296}%
\special{pa 2553 3317}%
\special{pa 2525 3336}%
\special{pa 2497 3353}%
\special{pa 2468 3369}%
\special{pa 2439 3382}%
\special{pa 2409 3393}%
\special{pa 2380 3402}%
\special{pa 2351 3408}%
\special{pa 2322 3412}%
\special{pa 2292 3414}%
\special{pa 2263 3414}%
\special{pa 2234 3411}%
\special{pa 2205 3407}%
\special{pa 2177 3401}%
\special{pa 2148 3393}%
\special{pa 2119 3383}%
\special{pa 2090 3372}%
\special{pa 2062 3359}%
\special{pa 2033 3344}%
\special{pa 2004 3328}%
\special{pa 1976 3311}%
\special{pa 1947 3293}%
\special{pa 1919 3274}%
\special{pa 1890 3254}%
\special{pa 1834 3210}%
\special{pa 1805 3187}%
\special{pa 1777 3164}%
\special{pa 1749 3140}%
\special{pa 1720 3115}%
\special{pa 1636 3040}%
\special{pa 1607 3014}%
\special{pa 1591 2999}%
\special{fp}%
%
\special{pn 8}%
\special{pa 2781 3009}%
\special{pa 2381 3009}%
\special{fp}%
%
\special{pn 8}%
\special{pa 2391 2999}%
\special{pa 2391 3399}%
\special{fp}%
%
\special{pn 8}%
\special{pa 2791 2999}%
\special{pa 2869 2909}%
\special{pa 2944 2819}%
\special{pa 2968 2789}%
\special{pa 2991 2760}%
\special{pa 3014 2730}%
\special{pa 3036 2700}%
\special{pa 3057 2671}%
\special{pa 3077 2642}%
\special{pa 3096 2612}%
\special{pa 3114 2583}%
\special{pa 3130 2554}%
\special{pa 3146 2526}%
\special{pa 3159 2497}%
\special{pa 3172 2469}%
\special{pa 3182 2441}%
\special{pa 3191 2413}%
\special{pa 3198 2385}%
\special{pa 3204 2357}%
\special{pa 3207 2330}%
\special{pa 3208 2303}%
\special{pa 3207 2276}%
\special{pa 3204 2250}%
\special{pa 3198 2224}%
\special{pa 3191 2198}%
\special{pa 3180 2172}%
\special{pa 3167 2147}%
\special{pa 3152 2122}%
\special{pa 3135 2098}%
\special{pa 3116 2074}%
\special{pa 3095 2051}%
\special{pa 3072 2028}%
\special{pa 3048 2006}%
\special{pa 3022 1985}%
\special{pa 2995 1965}%
\special{pa 2966 1945}%
\special{pa 2937 1927}%
\special{pa 2906 1909}%
\special{pa 2875 1892}%
\special{pa 2842 1877}%
\special{pa 2810 1862}%
\special{pa 2776 1849}%
\special{pa 2743 1837}%
\special{pa 2709 1826}%
\special{pa 2675 1816}%
\special{pa 2641 1808}%
\special{pa 2607 1802}%
\special{pa 2574 1797}%
\special{pa 2541 1793}%
\special{pa 2508 1791}%
\special{pa 2475 1790}%
\special{pa 2443 1790}%
\special{pa 2411 1792}%
\special{pa 2379 1795}%
\special{pa 2317 1805}%
\special{pa 2286 1812}%
\special{pa 2255 1820}%
\special{pa 2225 1829}%
\special{pa 2165 1849}%
\special{pa 2135 1861}%
\special{pa 2106 1874}%
\special{pa 2076 1887}%
\special{pa 2047 1901}%
\special{pa 2018 1916}%
\special{pa 1960 1948}%
\special{pa 1932 1965}%
\special{pa 1903 1982}%
\special{pa 1847 2018}%
\special{pa 1818 2037}%
\special{pa 1762 2075}%
\special{pa 1678 2135}%
\special{pa 1651 2155}%
\special{pa 1623 2176}%
\special{pa 1595 2196}%
\special{pa 1591 2199}%
\special{fp}%
%
\special{pn 8}%
\special{pa 2400 3400}%
\special{pa 2367 3417}%
\special{pa 2334 3435}%
\special{pa 2302 3452}%
\special{pa 2269 3468}%
\special{pa 2236 3485}%
\special{pa 2204 3501}%
\special{pa 2171 3516}%
\special{pa 2139 3531}%
\special{pa 2075 3559}%
\special{pa 2011 3583}%
\special{pa 1980 3593}%
\special{pa 1948 3603}%
\special{pa 1917 3611}%
\special{pa 1886 3618}%
\special{pa 1856 3624}%
\special{pa 1825 3628}%
\special{pa 1795 3630}%
\special{pa 1765 3631}%
\special{pa 1736 3630}%
\special{pa 1707 3628}%
\special{pa 1678 3623}%
\special{pa 1649 3617}%
\special{pa 1621 3608}%
\special{pa 1594 3597}%
\special{pa 1566 3584}%
\special{pa 1540 3569}%
\special{pa 1513 3552}%
\special{pa 1488 3534}%
\special{pa 1463 3513}%
\special{pa 1438 3491}%
\special{pa 1415 3467}%
\special{pa 1392 3442}%
\special{pa 1371 3416}%
\special{pa 1350 3388}%
\special{pa 1330 3360}%
\special{pa 1294 3300}%
\special{pa 1278 3269}%
\special{pa 1263 3237}%
\special{pa 1250 3205}%
\special{pa 1238 3173}%
\special{pa 1227 3140}%
\special{pa 1218 3108}%
\special{pa 1211 3075}%
\special{pa 1205 3042}%
\special{pa 1201 3010}%
\special{pa 1199 2978}%
\special{pa 1198 2946}%
\special{pa 1199 2915}%
\special{pa 1202 2884}%
\special{pa 1207 2853}%
\special{pa 1213 2823}%
\special{pa 1220 2793}%
\special{pa 1229 2763}%
\special{pa 1239 2733}%
\special{pa 1250 2704}%
\special{pa 1276 2646}%
\special{pa 1291 2618}%
\special{pa 1307 2589}%
\special{pa 1323 2561}%
\special{pa 1341 2533}%
\special{pa 1359 2506}%
\special{pa 1378 2478}%
\special{pa 1398 2450}%
\special{pa 1418 2423}%
\special{pa 1460 2369}%
\special{pa 1482 2341}%
\special{pa 1504 2314}%
\special{pa 1526 2288}%
\special{pa 1549 2261}%
\special{pa 1571 2234}%
\special{pa 1594 2207}%
\special{pa 1600 2200}%
\special{fp}%
%
\special{sh 0}%
\special{ia 2380 3400 41 41 0.0000000 6.2831853}%
\special{pn 8}%
\special{ar 2380 3400 41 41 0.0000000 6.2831853}%
%
\special{sh 0}%
\special{ia 2780 3000 41 41 0.0000000 6.2831853}%
\special{pn 8}%
\special{ar 2780 3000 41 41 0.0000000 6.2831853}%
\end{picture}}%

\caption{The plane graph $G_0$ (white vertices indicate $V_4$ and black vertices indicate $V_{\ge 5}$).}
\label{fig_2}
\end{center}
\end{figure}
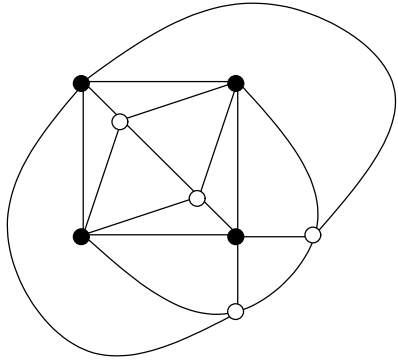

We define two transformations as depicted in Figure~\ref{fig_3}.
Let $C=c_1c_2c_3c_4c_1$ be a $4$-cycle such that the interior of $C$ contains 
exactly two vertices $d_1, d_2 \in V_4$ and $c_i \in V_{\ge 5}$ for each $i$.
By symmetry, we suppose that 
$\{ c_1d_1, c_2d_1, c_2d_2, c_3d_2, c_4d_1, c_4d_2, d_1d_2 \} \subset E(G)$.
The {\sl extension 1} is the vertex splitting $d_2$ into $d'_2, d''_2$ such that 
$d'_2, d''_2$ become vertices of degree $4$ and $d_1$ becomes a vertex of degree $5$
in the resulting graph.
The inverse operation of the extension 1 is called the {\sl contraction 1}.
Note that we can apply contraction 1 only if $c_3$ has degree at least $6$.

The {\sl extension 2} is the vertex splitting $c_2$ into $c'_2, c''_2$ such that 
$c'_2$ becomes a vertex of degree at least $5$ and $c''_2$ becomes a vertex of degree $5$
in the resulting graph.
Note that we can apply the extension 2 when $c_2$ has degree at least $6$.
The inverse operation of the extension 2 is the {\sl contraction 2}.
Note that we can apply the contraction 2 only if both $c_1$ and $c_3$ have degree at least $6$.

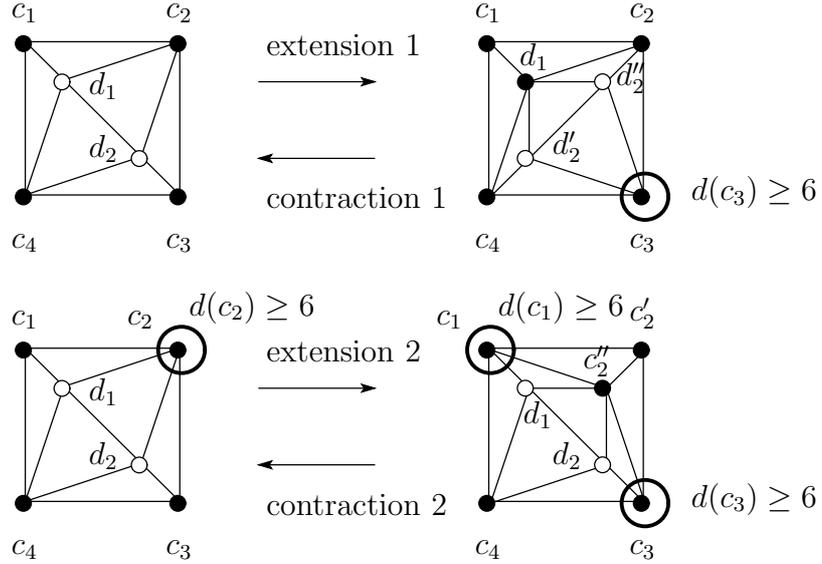
\begin{figure}[htb]
\begin{center}
{\unitlength 0.1in%
\begin{picture}(35.2000,28.0000)(7.2000,-29.6000)%
%
\special{pn 8}%
\special{pa 791 391}%
\special{pa 1591 1191}%
\special{fp}%
\special{pa 1591 1191}%
\special{pa 1591 1191}%
\special{fp}%
%
\special{pn 8}%
\special{pa 1591 1191}%
\special{pa 1591 391}%
\special{fp}%
\special{pa 1591 391}%
\special{pa 791 391}%
\special{fp}%
\special{pa 791 391}%
\special{pa 791 1191}%
\special{fp}%
\special{pa 791 1191}%
\special{pa 1591 1191}%
\special{fp}%
%
\special{sh 1.000}%
\special{ia 781 401 41 41 0.0000000 6.2831853}%
\special{pn 8}%
\special{ar 781 401 41 41 0.0000000 6.2831853}%
%
\special{sh 1.000}%
\special{ia 1581 1201 41 41 0.0000000 6.2831853}%
\special{pn 8}%
\special{ar 1581 1201 41 41 0.0000000 6.2831853}%
%
\special{sh 1.000}%
\special{ia 781 1201 41 41 0.0000000 6.2831853}%
\special{pn 8}%
\special{ar 781 1201 41 41 0.0000000 6.2831853}%
%
\special{sh 1.000}%
\special{ia 1581 401 41 41 0.0000000 6.2831853}%
\special{pn 8}%
\special{ar 1581 401 41 41 0.0000000 6.2831853}%
%
\special{pn 8}%
\special{pa 1591 391}%
\special{pa 991 591}%
\special{fp}%
\special{pa 991 591}%
\special{pa 791 1191}%
\special{fp}%
\special{pa 791 1191}%
\special{pa 1391 991}%
\special{fp}%
\special{pa 1391 991}%
\special{pa 1591 391}%
\special{fp}%
%
\special{sh 0}%
\special{ia 1381 1001 41 41 0.0000000 6.2831853}%
\special{pn 8}%
\special{ar 1381 1001 41 41 0.0000000 6.2831853}%
%
\special{sh 0}%
\special{ia 981 601 41 41 0.0000000 6.2831853}%
\special{pn 8}%
\special{ar 981 601 41 41 0.0000000 6.2831853}%
%
\special{pn 8}%
\special{pa 2000 600}%
\special{pa 2600 600}%
\special{fp}%
\special{sh 1}%
\special{pa 2600 600}%
\special{pa 2533 580}%
\special{pa 2547 600}%
\special{pa 2533 620}%
\special{pa 2600 600}%
\special{fp}%
\special{pa 2600 1000}%
\special{pa 2000 1000}%
\special{fp}%
\special{sh 1}%
\special{pa 2000 1000}%
\special{pa 2067 1020}%
\special{pa 2053 1000}%
\special{pa 2067 980}%
\special{pa 2000 1000}%
\special{fp}%
\put(20.4000,-4.6000){\makebox(0,0)[lb]{extension 1}}%
\put(20.4000,-12.6000){\makebox(0,0)[lb]{contraction 1}}%
%
\special{pn 8}%
\special{pa 3991 1191}%
\special{pa 3991 391}%
\special{fp}%
\special{pa 3991 391}%
\special{pa 3191 391}%
\special{fp}%
\special{pa 3191 391}%
\special{pa 3191 1191}%
\special{fp}%
\special{pa 3191 1191}%
\special{pa 3991 1191}%
\special{fp}%
%
\special{sh 1.000}%
\special{ia 3181 401 41 41 0.0000000 6.2831853}%
\special{pn 8}%
\special{ar 3181 401 41 41 0.0000000 6.2831853}%
%
\special{sh 1.000}%
\special{ia 3981 1201 41 41 0.0000000 6.2831853}%
\special{pn 8}%
\special{ar 3981 1201 41 41 0.0000000 6.2831853}%
%
\special{sh 1.000}%
\special{ia 3181 1201 41 41 0.0000000 6.2831853}%
\special{pn 8}%
\special{ar 3181 1201 41 41 0.0000000 6.2831853}%
%
\special{sh 1.000}%
\special{ia 3981 401 41 41 0.0000000 6.2831853}%
\special{pn 8}%
\special{ar 3981 401 41 41 0.0000000 6.2831853}%
%
\special{sh 1.000}%
\special{ia 3381 601 41 41 0.0000000 6.2831853}%
\special{pn 8}%
\special{ar 3381 601 41 41 0.0000000 6.2831853}%
%
\special{pn 8}%
\special{pa 3200 400}%
\special{pa 3400 600}%
\special{fp}%
\special{pa 3400 600}%
\special{pa 3200 1200}%
\special{fp}%
\special{pa 4000 400}%
\special{pa 3400 600}%
\special{fp}%
%
\special{pn 8}%
\special{pa 3200 1200}%
\special{pa 4000 400}%
\special{fp}%
%
\special{pn 8}%
\special{pa 3400 600}%
\special{pa 3800 600}%
\special{fp}%
\special{pa 3800 600}%
\special{pa 4000 1200}%
\special{fp}%
\special{pa 4000 1200}%
\special{pa 3400 1000}%
\special{fp}%
\special{pa 3400 1000}%
\special{pa 3400 600}%
\special{fp}%
%
\special{sh 0}%
\special{ia 3380 1000 41 41 0.0000000 6.2831853}%
\special{pn 8}%
\special{ar 3380 1000 41 41 0.0000000 6.2831853}%
%
\special{sh 0}%
\special{ia 3780 600 41 41 0.0000000 6.2831853}%
\special{pn 8}%
\special{ar 3780 600 41 41 0.0000000 6.2831853}%
%
\special{pn 8}%
\special{pa 791 1991}%
\special{pa 1591 2791}%
\special{fp}%
\special{pa 1591 2791}%
\special{pa 1591 2791}%
\special{fp}%
%
\special{pn 8}%
\special{pa 1591 2791}%
\special{pa 1591 1991}%
\special{fp}%
\special{pa 1591 1991}%
\special{pa 791 1991}%
\special{fp}%
\special{pa 791 1991}%
\special{pa 791 2791}%
\special{fp}%
\special{pa 791 2791}%
\special{pa 1591 2791}%
\special{fp}%
%
\special{sh 1.000}%
\special{ia 781 2001 41 41 0.0000000 6.2831853}%
\special{pn 8}%
\special{ar 781 2001 41 41 0.0000000 6.2831853}%
%
\special{sh 1.000}%
\special{ia 1581 2801 41 41 0.0000000 6.2831853}%
\special{pn 8}%
\special{ar 1581 2801 41 41 0.0000000 6.2831853}%
%
\special{sh 1.000}%
\special{ia 781 2801 41 41 0.0000000 6.2831853}%
\special{pn 8}%
\special{ar 781 2801 41 41 0.0000000 6.2831853}%
%
\special{sh 1.000}%
\special{ia 1581 2001 41 41 0.0000000 6.2831853}%
\special{pn 8}%
\special{ar 1581 2001 41 41 0.0000000 6.2831853}%
%
\special{pn 8}%
\special{pa 1591 1991}%
\special{pa 991 2191}%
\special{fp}%
\special{pa 991 2191}%
\special{pa 791 2791}%
\special{fp}%
\special{pa 791 2791}%
\special{pa 1391 2591}%
\special{fp}%
\special{pa 1391 2591}%
\special{pa 1591 1991}%
\special{fp}%
%
\special{sh 0}%
\special{ia 1381 2601 41 41 0.0000000 6.2831853}%
\special{pn 8}%
\special{ar 1381 2601 41 41 0.0000000 6.2831853}%
%
\special{sh 0}%
\special{ia 981 2201 41 41 0.0000000 6.2831853}%
\special{pn 8}%
\special{ar 981 2201 41 41 0.0000000 6.2831853}%
%
\special{pn 8}%
\special{pa 2000 2200}%
\special{pa 2600 2200}%
\special{fp}%
\special{sh 1}%
\special{pa 2600 2200}%
\special{pa 2533 2180}%
\special{pa 2547 2200}%
\special{pa 2533 2220}%
\special{pa 2600 2200}%
\special{fp}%
\special{pa 2600 2600}%
\special{pa 2000 2600}%
\special{fp}%
\special{sh 1}%
\special{pa 2000 2600}%
\special{pa 2067 2620}%
\special{pa 2053 2600}%
\special{pa 2067 2580}%
\special{pa 2000 2600}%
\special{fp}%
\put(20.4000,-20.6000){\makebox(0,0)[lb]{extension 2}}%
\put(20.4000,-28.6000){\makebox(0,0)[lb]{contraction 2}}%
%
\special{pn 8}%
\special{pa 3991 2791}%
\special{pa 3991 1991}%
\special{fp}%
\special{pa 3991 1991}%
\special{pa 3191 1991}%
\special{fp}%
\special{pa 3191 1991}%
\special{pa 3191 2791}%
\special{fp}%
\special{pa 3191 2791}%
\special{pa 3991 2791}%
\special{fp}%
%
\special{sh 1.000}%
\special{ia 3181 2001 41 41 0.0000000 6.2831853}%
\special{pn 8}%
\special{ar 3181 2001 41 41 0.0000000 6.2831853}%
%
\special{sh 1.000}%
\special{ia 3981 2801 41 41 0.0000000 6.2831853}%
\special{pn 8}%
\special{ar 3981 2801 41 41 0.0000000 6.2831853}%
%
\special{sh 1.000}%
\special{ia 3181 2801 41 41 0.0000000 6.2831853}%
\special{pn 8}%
\special{ar 3181 2801 41 41 0.0000000 6.2831853}%
%
\special{sh 1.000}%
\special{ia 3981 2001 41 41 0.0000000 6.2831853}%
\special{pn 8}%
\special{ar 3981 2001 41 41 0.0000000 6.2831853}%
%
\special{pn 8}%
\special{pa 3200 2000}%
\special{pa 3800 2200}%
\special{fp}%
\special{pa 3800 2200}%
\special{pa 4000 2800}%
\special{fp}%
\special{pa 4000 2800}%
\special{pa 3200 2000}%
\special{fp}%
%
\special{pn 8}%
\special{pa 3800 2200}%
\special{pa 3400 2200}%
\special{fp}%
\special{pa 3400 2200}%
\special{pa 3200 2800}%
\special{fp}%
\special{pa 3200 2800}%
\special{pa 3800 2600}%
\special{fp}%
\special{pa 3800 2600}%
\special{pa 3800 2200}%
\special{fp}%
%
\special{pn 8}%
\special{pa 3800 2200}%
\special{pa 4000 2000}%
\special{fp}%
%
\special{sh 1.000}%
\special{ia 3780 2200 41 41 0.0000000 6.2831853}%
\special{pn 8}%
\special{ar 3780 2200 41 41 0.0000000 6.2831853}%
%
\special{sh 0}%
\special{ia 3380 2200 41 41 0.0000000 6.2831853}%
\special{pn 8}%
\special{ar 3380 2200 41 41 0.0000000 6.2831853}%
%
\special{sh 0}%
\special{ia 3780 2600 41 41 0.0000000 6.2831853}%
\special{pn 8}%
\special{ar 3780 2600 41 41 0.0000000 6.2831853}%
%
\special{pn 20}%
\special{ar 4000 1200 121 121 0.0000000 6.2831853}%
\put(42.4000,-12.6000){\makebox(0,0)[lb]{$d(c_3) \ge 6$}}%
%
\special{pn 20}%
\special{ar 3200 2000 121 121 0.0000000 6.2831853}%
%
\special{pn 20}%
\special{ar 4000 2800 121 121 0.0000000 6.2831853}%
\put(32.4000,-18.6000){\makebox(0,0)[lb]{$d(c_1) \ge 6$}}%
\put(42.4000,-28.6000){\makebox(0,0)[lb]{$d(c_3) \ge 6$}}%
%
\special{pn 20}%
\special{ar 1600 2000 121 121 0.0000000 6.2831853}%
\put(16.4000,-18.6000){\makebox(0,0)[lb]{$d(c_2) \ge 6$}}%
\put(7.2000,-2.9000){\makebox(0,0)[lb]{$c_1$}}%
\put(15.2000,-2.9000){\makebox(0,0)[lb]{$c_2$}}%
\put(31.2000,-2.9000){\makebox(0,0)[lb]{$c_1$}}%
\put(39.2000,-2.9000){\makebox(0,0)[lb]{$c_2$}}%
\put(29.2000,-18.9000){\makebox(0,0)[lb]{$c_1$}}%
\put(39.2000,-18.9000){\makebox(0,0)[lb]{$c'_2$}}%
\put(7.2000,-18.9000){\makebox(0,0)[lb]{$c_1$}}%
\put(13.2000,-18.9000){\makebox(0,0)[lb]{$c_2$}}%
\put(7.2000,-14.9000){\makebox(0,0)[lb]{$c_4$}}%
\put(15.2000,-14.9000){\makebox(0,0)[lb]{$c_3$}}%
\put(7.2000,-30.9000){\makebox(0,0)[lb]{$c_4$}}%
\put(15.2000,-30.9000){\makebox(0,0)[lb]{$c_3$}}%
\put(31.2000,-30.9000){\makebox(0,0)[lb]{$c_4$}}%
\put(39.2000,-30.9000){\makebox(0,0)[lb]{$c_3$}}%
\put(31.2000,-14.9000){\makebox(0,0)[lb]{$c_4$}}%
\put(39.2000,-14.9000){\makebox(0,0)[lb]{$c_3$}}%
\put(11.2000,-6.9000){\makebox(0,0)[lb]{$d_1$}}%
\put(11.2000,-10.2000){\makebox(0,0)[lb]{$d_2$}}%
\put(11.2000,-22.9000){\makebox(0,0)[lb]{$d_1$}}%
\put(11.2000,-26.2000){\makebox(0,0)[lb]{$d_2$}}%
\put(33.5000,-5.4000){\makebox(0,0)[lb]{$d_1$}}%
\put(35.2000,-10.2000){\makebox(0,0)[lb]{$d'_2$}}%
\put(38.5000,-6.6000){\makebox(0,0)[lb]{$d''_2$}}%
\put(33.7000,-24.1000){\makebox(0,0)[lb]{$d_1$}}%
\put(35.2000,-26.2000){\makebox(0,0)[lb]{$d_2$}}%
\put(36.8000,-21.4000){\makebox(0,0)[lb]{$c''_2$}}%
\end{picture}}%

\caption{Two transformations (white vertices indicate $V_4$ and black vertices indicate $V_{\ge 5}$ ).}
\label{fig_3}
\end{center}
\end{figure}

Let $\mathcal{G}$ be the set of plane triangulations obtained from $G_0$ by repeatedly applying extensions 1 or 2.
Remark that $G_0 \in \mathcal{G}$.

The following is our main theorem.

\begin{thm}\label{thm:main2}
Let $G$ be a $4$-connected plane triangulation of order at least $7$.
Then we have $|E_4| \ge |V_{\ge 5}|+2$. 
Moreover, we have $|E_4| = |V_{\ge 5}|+2$ 
if and only if $G \in \mathcal{G}$. 
\end{thm}

Let ${\rm DW}_n$ denote the double wheel of order $n+2$, and let ${\rm DW}^{-}_n$ denote a graph obtained from ${\rm DW}_n$ by deleting an edge $x_1x_2$ such that $d_{{\rm DW}_n}(x_i)=4$ for each $i \in \{ 1,2 \}$.  

Since ${\rm DW}_4$ does not contain any $4$-contractible edges,
we assume $|G|\geq 7$ in Theorem~\ref{thm:main2}.
Remark that ${\rm DW}_4$ is the unique $4$-connected plane triangulation of order $6$.


\section{Proof of Theorem~\ref{thm:main2}}
\label{sec:main2}
First, we introduce a lemma that is needed in our proofs.

For a connected graph $G$, a cycle $C$ is called a {\sl separating cycle} if $G-V(C)$ contains at least two components.

\begin{lem}\label{lem:2.1}
Let $e$ be an edge of a $4$-connected plane triangulation $G$.
Then $e \notin E_4$ if and only if 
$e$ is contained in a separating $4$-cycle.
\end{lem}
\proof
If $e$ is contained in a separating $4$-cycle, then it is clear that $e \notin E_4$.
Thus, if part is proved. 
Now we prove only if part.
Let $G'$ be a plane triangulation 
obtained from $G$ by contracting $e$.
Since $e \notin E_4$, $G'$ is not $4$-connected, and hence $G'$ contains a separating $3$-cycle $C$.
If $C$ is contained in $G$, then $G$ is not $4$-connected, a contradiction. 
So, $C$ is corresponding to a separating $4$-cycle in $G$ containing $e$.
\qed

Now we prove Theorem~\ref{thm:main2}.

\noindent
{\bf Proof of Theorem~\ref{thm:main2}.}\\
Since both transformations (extensions 1 and 2) create one new vertex of $V_{\ge 5}$ and 
one new edge of $E_4$, we can see that if $G \in \mathcal{G}$, then we have $|E_4| = |V_{\ge 5}|+2$.
So it suffices to show that 
if $G$ is a $4$-connected plane triangulation of order at least $7$, 
then we have either $|E_4| > |V_{\ge 5}|+2$ or $G \in \mathcal{G}$. 

Let $G$ be a $4$-connected plane triangulation.
If $|G|=7$, then $G$ is isomorphic to ${\rm DW}_{5}$, 
and hence we have $|E_4| = 5 > |V_{\ge 5}|+2$.
Thus, we suppose $|G| \ge 8$.
By way of contradiction, we choose $G$ so that 
\begin{enumerate}[(1)]
\item $|E_4| \le |V_{\ge 5}|+2$, 
\item $G \notin \mathcal{G}$, and 
\item $|G|$ is as small as possible subject to (1) and (2). 
\end{enumerate}


Let $W_{n}$ denote the wheel of order $n+1$.
Since $G$ is a $4$-connected triangulation, for each $v \in V(G)$, $G[N_G(v)]$ is the induced cycle, which is called the {\sl link} of $v$.
For a cycle $C=c_1c_2 \ldots c_nc_1$, 
let $\overrightarrow{C}$ denote an {\sl orientation} of $C$, 
and let $l_i\overrightarrow{C}l_j$ denote the {\sl path from $l_i$ to $l_j$} on $\overrightarrow{C}$.
Let $v$ be a vertex of degree $p \ge 5$.
Since $G$ is $4$-connected, $G[N_G[v]]$ is isomorphic to $W_{p}$.
In $E(G[N_G[v]]) \cap E_4$,
\begin{enumerate}
\item let $F_{v,1}= \{ uv \in E_4 : u \in V_{\ge 5} \}$, 
\item let $F_{v,2}= \{ uv \in E_4 : u \in V_{4} \}$, and
\item let $F_{v,3}= \{ uu' \in E_4 : u, u' \in N_G(v) \cap V_{4} \}$.
\end{enumerate}


For each $i \in \{ 1, 2, 3 \}$,
let $w_i(v)$ be the number of edges in $F_{v, i}$.
We let 
\[
w(v)=w_1(v)+2w_2(v)+w_3(v).
\]

\begin{claim}\label{lem:4_0}
For $e\in E_4$, $|\{v\in V_{\geq 5}:e\in F_{v,i}$  for some $ i \}|\leq 2$.
Furthermore, if $e\in F_{v,2}$ for some $v\in V_{\geq 5}$, then $e\notin F_{x,i}$ for any $x\in V_{\geq 5}-\{v\}$ and any $i$.
\end{claim}
\proof
Suppose $v \in V_{\ge 5}$.
Suppose that $vu \in F_{v,1}$. Then $vu \in F_{u,1}$.
Thus $vu \notin F_{x, i}$ for any $i \in \{1, 2\}$ and $x \in V_{\ge 5} - \{ v, u\}$.
Moreover $vu \notin F_{x, 3}$ for any $x \in V_{\ge 5}$ because $d_G(v) \ge 5$.

Suppose that $vu \in F_{v,2}$. Then $vu \notin F_{x, 1}$ for any $x \in V_{\ge 5}-\{ v \}$
because $d_G(u)=4$.
Moreover $vu \notin F_{x, i}$ for any $i \in \{2, 3\}$ and $x \in V_{\ge 5} - \{ v \}$ 
because $d_G(v) \ge 5$.

Suppose that $uu' \in F_{v, 3}$. Then $uu' \notin F_{x, i}$ 
for any $i \in \{1, 2\}$ and  $x \in V_{\ge 5}$ 
because $d_G(u)=d_G(u')=4$.
Let $v'uu'$ be the triangular face such that $v \neq v'$. 
Although $uu'$ may be contained in $F_{v' , 3}$, 
$uu' \notin F_{x, 3}$ for $x \in V_{\ge 5} - \{ v, v' \}$
because $G$ is a plane triangulation.
\qed

\begin{claim}\label{lem:4_1}
Let $k$ be a nonnegative integer.
If $\sum_{v \in V_{\ge 5}}w(v) \ge 2|V_{\ge 5}|+k$, then $|E_4| \ge |V_{\ge 5}|+\frac{k}{2}$.
\end{claim}
\proof
By Claim~\ref{lem:4_0}, for each $e \in E_4$, $|\{v\in V_{\geq 5}:e\in F_{v,i}$  for some $ i \}|\leq 2$. 
Moreover,  if $e \in F_{v,2}$, then $e\notin F_{x,i}$ for any $x\in V_{\geq 5}-\{v\}$ and any $i$.
Thus we have 
\[
2|E_4| \ge \sum_{v \in V_{\ge 5}}w(v) \ge 2|V_{\ge 5}|+k.
\]
So, $|E_4| \ge |V_{\ge 5}|+\frac{k}{2}$.
\qed

By the choice of $G$ and Claim~\ref{lem:4_1}, 
we have $\sum_{v \in V_{\ge 5}}w(v) \le 2|V_{\ge 5}|+4$.

For a fixed $v \in V_{\ge 5}$, 
let $L=l_1l_2 \ldots l_pl_1$ ($p \ge 5$) be the link of $v$.
A separating $4$-cycle $C$ containing $v, l_i, l_j$ 
($|i -j | \ge 2$ and $\{ i, j\} \neq \{ 1, p \}$) is called a
{\sl $k$-jump separating $4$-cycle on $v$} 
if either $|l_i\overrightarrow{L}l_j|=k+2$ or $|l_j\overrightarrow{L}l_i|=k+2$.
Moreover, if $|l_i\overrightarrow{L}l_j|=k+2$, then the region bounded by $C$ and containing $l_i\overrightarrow{L}l_j$ is called 
an {\sl inside region on $C$}.

\begin{claim}\label{lem:4_2}
Suppose $v \in V_{\ge 5}$.
For $k \ge 2$, if there exists a $k$-jump separating $4$-cycle $C$ on $v$, 
then the inside region of $C$ contains at least one edge $e \in F_{v,i}$ for some $i \in \{ 1, 2, 3\}$.
\end{claim}
\proof
If the inside region of $C$ contains an edge of $E_4$ incident to $v$, 
then we are done.
So, we suppose that there is no edge in $E_4$ incident to $v$.
In the inside region of $C$, take a $k'$-jump separating $4$-cycle $C'$ on $v$ so that
\begin{enumerate}[(i)]
\item $k' \ge 2$, and
\item subject to (i), the number of vertices in the inside region of $C'$ is as small as possible.
\end{enumerate}

\begin{subclaim}\label{cl:4_2_1}
$k'=2$.
\end{subclaim}
\proof
By way of contradiction, suppose $k' \ge 3$.
Suppose $C'=vl_iyl_j$ such that $j \ge i+4$ and $y \in V(G)  - N_G[v]$.
By the choice of $C'$, $yl_{i+1} \notin E(G)$.
Since $vl_{i+1} \notin E_4$, $vl_{i+1}$ is contained in a $1$-jump separating $4$-cycle $vl_{i+1}y'l_{i+3}v$, 
where $y'$ is a vertex of $V(G) - (N_G[v] \cup \{ y \})$ in the inside region of $C'$.
By the Jordan curve theorem on $vl_{i+1}y'l_{i+3}v$,  we have $yl_{i+2} \notin E(G)$. 
This together with $vl_{i+2} \notin E_4$ implies that
$vl_{i+2}$ is contained in either a $1$-jump separating $4$-cycle $vl_{i+2}y'l_{i+4}v$
or a $1$-jump separating $4$-cycle $vl_{i+2}y'l_{i}v$.
Then we can find either a $2$-jump separating $4$-cycle $vl_{i+1}y'l_{i+4}v$
or a $2$-jump separating $4$-cycle $vl_{i}y'l_{i+3}v$, contrary to the choice of $C'$.
\qed

By Subclaim~\ref{cl:4_2_1}, we may assume that $C'=vl_1yl_4v$, where  $y \in V(G) - N_G[v]$.
By the choice of $C'$, we have $l_iy \in E(G)$ for each $i \in \{ 2, 3 \}$ 
(otherwise we can find another $2$-jump separating $4$-cycle in the inside region of $C'$).
Thus $G[\{ v, y, l_1, l_2, l_3, l_4\}]$ is ${\rm DW}^{-}_4$, and hence $l_2l_3 \in F_{v,3}$. 
\qed

\begin{claim}\label{lem:4_3}
For $v \in V_{\ge 5}$, suppose that there exists a separating $4$-cycle containing $v$. 
Then $G$ has a separating $4$-cycle $C=vl_iyl_jv$ such that 
both the interior and the exterior of $C$ contain at least one edge $e \in F_{v,i}$ for some $i \in \{ 1, 2, 3\}$.
\end{claim}
\proof
For a separating $4$-cycle $C=vl_iyl_jv$, 
let $k^{C}_1$ and $k^{C}_2$ are integers such that $|l_i\overrightarrow{L}l_j|=k^{C}_1+2$ and $|l_j\overrightarrow{L}l_i|=k^{C}_2+2$.
If there exists $C$ such that $k^{C}_i \ge 2$ for each $i \in\{ 1, 2 \}$, 
then we can take required edges by Claim~\ref{lem:4_2}.

So, we assume that for each $C$, either $k^{C}_1 =1$ or $k^{C}_2 =1$.
By symmetry, we may assume that $vl_1 \notin E_4$ and it is contained in 
a $1$-jump separating $4$-cycle $C=vl_1yl_3v$, where  $y \in V(G) - N_G[v]$.
By Claim~\ref{lem:4_2}, the inside region of $C$ containing $l_3\overrightarrow{L}l_1$ 
contains an edge $e \in F_{v,i}$ for some $i \in \{ 1, 2, 3\}$.
So, if $vl_2 \in E_4$, then we are done.
Thus, we suppose $vl_2 \notin E_4$.
By symmetry, we may assume that $vl_2$ is contained in 
a $1$-jump separating $4$-cycle $vl_2yl_4v$.
Then we have $l_2l_3, e \in F_{v, 3}$, and hence $C$ is a desired separating $4$-cycle.
\qed

By Claim~\ref{lem:4_3}, we can see that for each $v \in V_{\ge 5}$, 
$w(v) \ge 2$.

\begin{claim}\label{lem:4_5}
For $v \in V_{\ge 5}$, 
suppose that $l_1, \ldots , l_{k}$ $(k \ge 4)$ be vertices of the link of $v$ such that  
$l_{i}l_{i+1} \in F_{v,3}$ for each $i \in \{ 2, \ldots, k-2 \}$ and 
$d_G(l_i) \ge 5$ for each $i \in \{ 1, k \}$.
Then there exists $y \in V_{\ge 5} - N_G[v]$ such that 
$G[\{ v, y, l_1, \ldots l_{k} \}]$ is isomorphic to ${\rm DW}^{-}_k$.
\end{claim}
\proof
Let $y$ be the vertex of $N_G(l_2) - \{ v, l_1, l_3\}$.
Since $l_{i}l_{i+1} \in F_{v,3}$ for each $i \in \{ 2, \ldots, k-2 \}$,
$d_G(l_j)=4$ for each $j \in \{ 2, \ldots, k-1 \}$.
This implies that for each $i \in \{ 1, \ldots, k \}$, 
$l_j$ is adjacent to $y$.
Since $G$ is $4$-connected, $y \notin N_G[v]$.
If $k \ge 5$, then $d_G(y) \ge 5$.
If $k=4$, then we have $l_1l_4 \notin E(G)$ because $G$ is $4$-connected, 
and hence $d_G(y) \ge 5$.
In either case, we have $y \in V_{\ge 5}$.
Since $G$ is $4$-connected, $vy \notin E(G)$.

Thus, it suffices to show that $l_1l_k \notin E(G)$.
Suppose that $d_G(v) = k$.
Since $G$ is $4$-connected, $yl_1l_ky$ be a facial cycle of $G$.
Then $G$ is ${\rm DW}_k$, contrary to the fact $d_G(l_i) \ge 5$ for each $i \in \{ 1, k \}$.
Thus $d_G(v) \ge k+1$. Since $G$ is $4$-connected, this leads to $l_{1}l_{k} \notin E(G)$ as desired.
\qed

Suppose $v \in V_{\ge 5}$ with $w(v) = 2$.
By Claim~\ref{lem:4_3}, there are two edges $e_1, e_2 \in F_{v,1} \cup F_{v,2} \cup F_{v,3}$.
For each $i$, if $e_i \in F_{v,1}$ such that end vertices of $e_i$ other than $v$ are 
not adjacent, then $v$ is said to be {\sl type(a)}.  
If $e_i \in F_{v,1}$ and $e_{3-i} \in F_{v,3}$ such that end vertices of $e_i$ other than $v$ are 
not adjacent to end vertices of $e_{3-i}$, then $v$ is said to be {\sl type(b)}.  
For each $i$, if $e_i \in F_{v,3}$ such that end vertices of $e_i$ are 
distinct and not adjacent, then $v$ is said to be {\sl type(c)}.

\begin{claim}\label{lem:4_6}
For $v \in V_{\ge 5}$, 
if $w(v) = 2$, then $v$ is type(a), type(b) or type(c).
\end{claim}
\proof
Since $w(v)=2$, it follows from Claim~\ref{lem:4_3} that $G$ has a separating $4$-cycle $C=vl_iyl_jv$ such that 
$|i-j| \ge 2, \{ i, j\} \neq \{ 1, p \}$ and both inside regions on $C$ contain at least one edge $F_{v,1} \cup F_{v,2} \cup F_{v,3}$ 
, say $e_1$ and $e_2$.
Since $w(v)=2$, $e_i \in F_{v,1} \cup F_{v,3}$ for each $i$.
Since $E(C) \cap E_4=\emptyset$, if $e_i \in F_{v,1}$ for each $i$, 
then $v$ is type(a). 

If $e_i \in F_{v,3}$ for some $i$, 
then $e_i$ is contained in $H$ isomorphic to ${\rm DW}^{-}_4$
by Claim~\ref{lem:4_5}. 
Since $e_{3-i} \in F_{v,1} \cup F_{v,3}$, $e_{3-i} \notin E(H)$, 
and hence $v$ is either type(b) or type(c).
\qed

\begin{claim}\label{lem:K_3-free}
Suppose $v_1, v_2, v_3 \in V_4$.
Then $G[\{ v_1, v_2, v_3 \}]$ is not isomorphic to $K_3$.
\end{claim}
\proof
By way of contradiction, $G[\{v_1, v_2, v_3\}]$ is isomorphic to $K_3$.
Since $v_1 \in V_4$, $N_G(v_1)$ induces a $4$-cycle $v_2v'_1v''_1v_3v_2$, 
where $v'_1, v''_1 \in N_G(v_1) - \{ v_2, v_3 \}$.
Since $v_2 \in V_4$, $N_G(v_2)$ induces a $4$-cycle $v_1v'_1v'_2v_3v_1$, 
where $v'_2 \in N_G(v_2) - \{ v_1, v_3 \}$.
Note that $v''_1 \neq v'_2$ because $G$ is $4$-connected.
Since $v_3 \in V_4$, $N_G(v_3)$ induces a $4$-cycle $v_1v_2v'_2v''_1v_1$, 
If $G$ has a vertex $V(G) - \{ v_1, v'_1, v''_1, v_2, v'_2, v_3 \} \neq \emptyset$, then $v'_1v''_1v'_2v'_1$ is a separating $3$-cycle, contrary to the fact $G$ is $4$-connected.
Thus  $V(G) = \{ v_1, v'_1, v''_1, v_2, v'_2, v_3 \} $, and hence $G$ is isomorphic to ${\rm DW}_4$, a contradiction.
\qed

\begin{claim}\label{cl2:1}
$V_4 \neq \emptyset$.
\end{claim}
\proof
By way of contradiction, suppose that $V_4 = \emptyset$.
Then we have $V_{\ge 5}=V(G)$.
If $G$ contains no separating $4$-cycle, then we have $|E(G)|=|E_4|=3|G|-6$ 
by Lemma~\ref{lem:2.1}.
Since $|G| \ge 8$, we have $3|G|-6 > |V_{\ge 5}|+2$, a contradiction.
Thus $G$ contains a separating $4$-cycle $C$.
We choose such a cycle so that the number of vertices contained in the interior of $C$ is as small as possible.
Since $C$ is a separating $4$-cycle, there is a vertex $v \in V_{\ge 5}$ in the interior of $C$.
By the minimality of $C$, all edges incident to $v$ are edges in $E_4$ 
because $v$ is contained in no separating $4$-cycle.
Thus, we have $w(v) \ge 5$. 
By symmetry of the interior and the exterior, we can find a vertex $u \in V_{\ge 5} - \{ v \}$ in the exterior of $C$ such that $w(u) \ge 5$.
This together with Claim~\ref{lem:4_3} implies that $\sum_{v \in V_{\ge 5}}w(v) \ge 2|V_{\ge 5}|+6$, 
contrary to the choice of $G$.
\qed

\begin{claim}\label{cl2:2}
Each component of $G[V_4]$ is isomorphic to a path of order at most two.
\end{claim}
\proof
If $G[V_4]$ has a vertex $v$ of degree at least $3$, then $G[N_{G[V_4]}[v]]$ contains a triangle since $G$ is a triangulation and $v \in V_4$, which contradicts Claim~\ref{lem:K_3-free}. Thus, the maximum degree of $G[V_4]$ is at most two. In particular, each component of $G[V_4]$ is isomorphic to either a path or a cycle.

\begin{subclaim}\label{cl2:2_2}
Each component of $G[V_4]$ is isomorphic to a path.
\end{subclaim}
\proof
By way of contradiction, suppose that $G[V_4]$ contains a cycle $C=c_{1}c_{2} \ldots c_{l}c_1$, where $l \ge 4$.
Since $C$ has no chord, each vertex on $C$ is adjacent to two vertices in $V(G)-V(C)$ 
such that one is contained in the interior of $C$ and the other is contained in the exterior of $C$.
Let $v$ and $u$ be vertices of $V(G)-V(C)$ which are adjacent to $c_{1}$.
Since $d_G(c_{1})=4$, we have $c_{2}v, c_{2}u \in E(G)$.
By the same argument, we have $c_{i}v, c_{i}u \in E(G)$ for each $i \in \{ 1, \ldots, l \}$.
Consequently, $G$ is isomorphic to ${\rm DW}_n$ for some $n \ge 5$. 
So, we have $|E_4| \ge |G|-2 > |V_{\ge 5}|+2$, contrary to the choice of $G$.
\qed

Now we show that each component of $G[V_4]$ is isomorphic to a path of order at most two.
By way of contradiction, suppose that $G[V_4]$ contains a path $P=p_{1}p_{2} \ldots p_{l}$, where $l \ge 3$.
Suppose $\{ p^{-}, v, u \} = N_G(p_1)-\{ p_{2} \}$ such that $p^{-}v p_{2} u p^{-}$ is a separating $4$-cycle.
Since $d_G(p_i)=4$ and $P$ contains no chord $p_{i}p_{i+2}$ for each $i$, 
we have $p_{i}v, p_{i}u \in E(G)$ for each $i \in \{ 1, \ldots, l \}$.
This implies that $G[V(P) \cup \{ p^{-}, p^{+}, v, u \}]$ is isomorphic to 
${\rm DW}_{l+1}$, ${\rm DW}_{l+2}$ or ${\rm DW}^{-}_{l+2}$, where $p^{+} \in N_G(p_{l}) - \{ p_{l-1}, v, u\}$.
Since $d_G(p^{-}), d_G(p^{+}) \ge 5$, $G[V(P) \cup \{ p^{-}, p^{+}, v, u \}]$ is isomorphic to 
neither ${\rm DW}_{l+1}$ nor ${\rm DW}_{l+2}$, and hence it is isomorphic to ${\rm DW}^{-}_{l+2}$.
Thus we have $d_G(v), d_G(u) \ge 6$.
Let $G'$ be a plane triangulation obtained from $G$ by contracting $p_{1}p_{2}$, 
and let $E'_4$ and $V'_{\ge 5}$ denote the set of $4$-contractible edges and the set of vertices of degree at least $5$ in $G'$, respectively.
By the definition of $G'$, we have $|E'_4|=|E_4|-1$.
Since $d_G(v), d_G(u) \ge 6$, we have $d_{G'}(v), d_{G'}(u) \ge 5$.
So, we have $|V_{\ge 5}|=|V'_{\ge 5}|$.
By the choice of $G$, we have $|E'_4| \ge |V'_{\ge 5}|+2$.
Then $|E_4| \ge  |V_{\ge 5}|+3 >  |V_{\ge 5}|+2$, contrary to the choice of $G$.
\qed

By Claim~\ref{cl2:2}, we may assume that 
each component of $G[V_4]$ is isomorphic to either $P_1$ or $P_2$, 
where $P_n$ denotes a path of order $n$.

\begin{claim}\label{cl2:3}
Let $C$ be a separating $4$-cycle of $G$ such that $V(C) \subset V_{\ge 5}$.
Then $C$ and the interior of $C$ contains one of the following;
\begin{enumerate}[{\rm (1)}]
\item a vertex $v \in V_{\ge 5} -V(C)$ with $w(v) \ge 5$, or
\item two vertices $v_1$ and $v_2$ with $w(v_1), w(v_2) \ge 3$. 
Moreover, if $v_i$ is contained in $V(C)$, then an edge of $E_4$ counted by $v_i$ as weight $2$ is contained in the interior of $C$.
\end{enumerate}
\end{claim}
\proof
In the interior of $C$, we choose a separating $4$-cycle $C'=c'_{1}c'_{2}c'_{3}c'_{4}c'_1$ with $V(C') \subset V_{\ge 5}$
so that the number of the vertices contained in the interior of $C'$ is as small as possible (we may have $C=C'$).
Suppose that the interior of $C'$ contains no vertices in $V_4$. 
Then it contains a vertex $v \in V_{\ge 5} -V(C)$. 
By the minimality of the interior of $C'$, each edge $e$ incident to $v$ 
is not contained in any separating $4$-cycle, and hence  $e \in E_4$.
This implies that (1) holds.

So, we consider the case when the interior of $C'$ contains a vertex $p_1 \in V_4$.
Suppose that $p_{1}$ is an induced $P_1$ in $G[V_4]$. 
Then $G[N_G[p_{1}]]$ is isomorphic to ${\rm DW}^{-}_3$ such that $N_G(p_{1}) \subset V_{\ge 5}$. 
This implies that $G[V(C') \cup \{ p_{1} \}]$ is
isomorphic to ${\rm DW}^{-}_3$, by the minimality of the interior of $C'$.
Since $|G| \ge 7$ (i.e., $G$ is not isomorphic to ${\rm DW}_4$), 
we can see that either $\{ c'_1p_{1}, c'_3p_{1} \} \subset E_4$ or 
$\{ c'_2p_{1}, c'_4p_{1} \} \subset E_4$.
By symmetry, we suppose that $\{ c'_1p_{1}, c'_3p_{1} \} \subset E_4$.
Then by Claim~\ref{lem:4_6}, (2) holds with $\{ c'_{1}, c'_{3} \} = \{ v_1, v_2 \}$.

Thus we may assume that there exist $p_{2} \in V_4$ such that $p_{1}p_{2} \in E(G)$.
Since $p_{1}, p_{2} \in V_4$, $G[N_G[\{p_{1}, p_{2}\}]]$ is isomorphic to ${\rm DW}^{-}_4$.
Note that $N_G(\{ p_{1}, p_{2} \}) \subset V_{\ge 5}$ by Claim~\ref{cl2:2}. This implies that $G[V(C') \cup \{ p_{1}, p_{2} \}]$ is
isomorphic to ${\rm DW}^{-}_4$, by the minimality of the interior of $C'$.
By symmetry, we may assume that $\{ c'_{2}, c'_{4} \} = N_G(p_{1}) \cap N_G(p_{2})$.
Then by Claim~\ref{lem:4_6}, (2) holds with $\{ c'_{1}, c'_{3} \} = \{ v_1, v_2 \}$.
\qed

\begin{claim}\label{cl2:4}
Each component of $G[V_4]$ is isomorphic to $P_2$.
\end{claim}
\proof
By way of contradiction, suppose that $G[V_4]$ contains a component of $P_1=p_{1}$.
Let $C=c_{1}c_{2}c_{3}c_{4}c_{1}$ be a separating $4$-cycle induced by $N_G(p_{1})$.
Suppose that $p_{1}$ is contained in no separating $4$-cycle.
Then $w(c_i) \ge 3$ for each $i \in \{ 1, 2, 3, 4\}$  by Claim~\ref{lem:4_6}.
By Claim~\ref{cl2:3}, 
$C$ and the exterior of $C$ contain vertices satisfying one of (1) or (2) 
in Claim~\ref{cl2:3}.
Hence by Claim~\ref{lem:4_2}, we have $\sum_{v \in V_{\ge 5}}w(v) \ge 2|V_{\ge 5}|+6$, 
contrary to the choice of $G$.

Therefore, $p_{1}$ is contained in a separating $4$-cycle.
By symmetry, we may assume that $p_{1}c_{2}yc_{4}p_{1}$ is such a separating $4$-cycle, 
where $y \in V(G)-(V(C) \cup \{ p_{1} \})$. 
Suppose that $y \in V_4$. 
By Claim~\ref{cl2:2}, there exists $y' \in (N_G(y)-\{ c_{2}, c_{4} \}) \cap V_{\ge 5}$.
If $y' \in \{ c_{1}, c_{3} \}$, then $p_{1}c_{2}yc_{4}p_{1}$ is not a separating $4$-cycle, a contradiction.
If $y'c_{i} \in E(G)$ for some $i \in \{ 1, 3 \}$, then $y'c_{i}c_{2}y'$ and $y'c_{i}c_{4}y'$ are triangular faces, 
and hence $y' \in V_4$, a contradiction.
Therefore, in both cases, $p_{1}c_{2}y'c_{4}p_{1}$ is a separating $4$-cycle such that all vertices are contained in $V_{\ge 5}$.
This fact implies that we can take a separating $4$-cycle $p_{1}c_{2}yc_{4}p_{1}$ so that
$y \in V_{\ge 5}$.
Then we can see that $w(c_i) \ge 3$ for each $i \in \{ 1, 3 \}$  by Claim~\ref{lem:4_6}.
For each $i \in \{ 1, 3 \}$, let $R_i$ is a region 
bounded by a separating $4$-cycle $c_{i}c_{2}yc_{4}c_{i}$ and not containing $p_{1}$.
By Claim~\ref{cl2:3}, the separating $4$-cycle and $R_i$ contains vertices satisfying one of (1) or (2) in Claim~\ref{cl2:3}.
Hence by Claim~\ref{lem:4_2}, we have have $\sum_{v \in V_{\ge 5}}w(v) \ge 2|V_{\ge 5}|+6$, 
contrary to the choice of $G$.
\qed

By Claims~\ref{cl2:1} and \ref{cl2:4}, $G$ contains $2k$ vertices of $V_4$ as an induced matching, where $k \ge 1$.
Suppose $p_{1}, p_{2} \in V_4$ with $p_{1}p_{2} \in E(G)$, $\{ x_{1}, x_{2} \} = N_G(p_{1}) \cap N_G(p_{2})$.
Let $p'_{i}$ be the vertex of $N_G(p_i) - \{ p_{3-i}, x_1, x_2 \}$ for each $i$.
Since $|G| \ge 8$, we have $p'_{1}p'_{2} \notin E(G)$ (otherwise $G$ is isomorphic to ${\rm DW}_4$).

\begin{claim}\label{cl2:5}
$d_G(x_{i})=5$ for some $i \in \{1, 2\}$.
\end{claim}
\proof
By way of contradiction, suppose that $d_G(x_{i}) \ge 6$ for each $i \in \{1, 2\}$.
Let $G'$ be a plane triangulation obtained from $G$ by contracting $p_{1}p_{2}$, 
and let $E'_4$ and $V'_{\ge 5}$ denote the set of $4$-contractible edges and the set of vertices of degree at least $5$ in $G'$, respectively.
We denote $p \in V(G')$ as a new vertex by the contraction of $p_{1}p_{2}$.
In $G'$, if $p'_{1}pp'_{2}$ (resp., $x_{1}px_{2}$) is contained in a separating $4$-cycle, 
then $x_{1}px_{2}$ (resp., $p'_{1}pp'_{2}$) is contained in no separating $4$-cycle.
Suppose that one of $p'_{1}pp'_{2}$ and $x_{1}px_{2}$ is contained in a separating $4$-cycle. 
Then we have $|E'_4|=|E_4|-1$.
On the other hand for each $i \in \{ 1, 2\}$, we have $d_{G'}(x_i) \ge 5$ since $d_G(x_i) \ge 6$, and hence
$|V_{\ge 5}|=|V'_{\ge 5}|$.
Furthermore, it follows from the choice $G$ that $|E'_4| \ge |V'_{\ge 5}|+2$. 
Consequently, $|E_4| =|E'_4|+1 \ge  |V'_{\ge 5}|+3 >  |V_{\ge 5}|+2$, which contradicts the choice of $G$.

Thus neither $p'_{1}pp'_{2}$ nor $x_{1}px_{2}$ is contained in a separating $4$-cycle.
This implies that $N_G(p'_{1}) \cap N_G(p'_{2})= \{ x_{1}, x_{2} \}$ 
and $N_G(x_{1}) \cap N_G(x_{2})=\{ p_{1}, p_{2}, p'_{1}, p'_{2} \}$.

For each $i \in \{ 1, 2 \}$, let $l^{i}_{1} \ldots l^{i}_{k_i}$ be the link of $x_i$ other than $p_{1}, p_{2}, p'_{1}, p'_{2}$, where $l^{i}_1$ is adjacent to $p'_{1}$ on the link. 

First we suppose that for each $i \in \{ 1, 2 \}$, $x_{i}l^{i}_{1},  \ldots, x_{i}l^{i}_{k_i} \in E_4$.
Then we have $\sum_{v \in V_{\ge 5}}w(v) \ge 2|V_{\ge 5}|+4$ 
because $w(p'_{i}) \ge 3$ and $w(x_{i}) \ge 3$ for each $i$.
For some $i$, if either $d_G(x_{i}) \ge 7$ or 
$\{ l^{i}_{1}, \ldots , l^{i}_{k_i} \} \cap V_4 \neq \emptyset$, 
then we have $w(x_{i}) \ge 4$.
This implies that $\sum_{v \in V_{\ge 5}}w(v) \ge 2|V_{\ge 5}|+5$, 
contrary to the fact $|E_4|=V_{\ge 5}+2$.
Thus we have $d_G(x_{i}) = 6$ and $\{ l^{i}_{1}, \ldots , l^{i}_{k_i} \} \cap V_4 = \emptyset$. 
By Claim~\ref{cl2:3} for $p'_{1}x_{1}p'_{2}x_{2}p'_{1}$, 
the exterior of $p'_{1}x_{1}p'_{2}x_{2}p'_{1}$ contains (1) or (2).
If it contains (2), then $\{ v, u \} \cup \{ x_1, x_2 \} = \emptyset$ because all edges of $E_4$ adjacent $x_i$ is counted by weight $1$.
Thus we have $\sum_{v \in V_{\ge 5}}w(v) \ge 2|V_{\ge 5}|+6$, contrary to the choice of $G$.

Next we suppose that for some $i \in \{ 1, 2 \}$, 
$\{ x_{i}l^{i}_{1},  \ldots, x_{i}l^{i}_{k_i} \} \not \subset E_4$.
By symmetry, we suppose that $x^{1}l^{1}_j \notin E_4$ for some $j \in \{ 1, \ldots, k_1 \}$.
Since $N_G(x_{1}) \cap N_G(x_{2})=\{ p_{1}, p_{2}, p'_{1}, p'_{2} \}$, 
$x^{1}l^{1}_j$ is contained in a separating $4$-cycle $x_{1}l^{1}_j y l' x_{1}$, 
where $y \in V(G)-N_G(x_{1})$ and $l' \in N_G(x_i) -\{ l^{1}_j, l^{1}_{j-1}, l^{1}_{j+1}, p_{1}, p_{2} \}$.
By Claim~\ref{cl2:3} for $x_{1}l^{1}_j y l' x_{1}$, 
the interior of $x_{1}l^{1}_j y l' x_{1}$ (i.e., region not containing $p_{1}$) contains (1) or (2).
This together with the fact $w(p'_{i}) \ge 3$ for each $i \in \{ 1, 2 \}$,
we have $\sum_{v \in V_{\ge 5}}w(v) \ge 2|V_{\ge 5}|+4$.
If $\{ x_{2}l^{2}_{1},  \ldots, x_{2}l^{2}_{k_2} \} \not \subset E_4$, then we repeat the argument.
If $\{ x_{2}l^{2}_{1},  \ldots, x_{2}l^{2}_{k_2} \} \subset E_4$, then we apply the above argument.
In both case, we have $\sum_{v \in V_{\ge 5}}w(v) \ge 2|V_{\ge 5}|+6$, 
contrary to the choice of $G$.
\qed

By Claim~\ref{cl2:5} and symmetry, we may assume that $d_G(x_{1})=5$.
If $d_G(x_{2}) \ge 6$, then we can apply the contraction $1$ to $p_{1}p_{2}$, 
contrary to the choice of $G$.
So, we consider the case when $d_G(x_{2})=5$.
For each $i \in \{ 1, 2 \}$, let $y_i$ be the vertex of $N_G(x_i) - \{ p_{1}, p_{2}, p'_{1}, p'_{2} \}$.
If $y_{1}y_{2} \in E(G)$, then $G$ is isomorphic to $G_0$.
Thus $y_{1}y_{2} \notin E(G)$.
Then we have $d_G(p'_{i}) \ge 6$ for each $i \in \{ 1, 2 \}$, 
and hence we can apply the contraction $2$ to $x_{1}y_{1}$, contrary to the choice of $G$.
\qed


\section*{Acknowledgment}

The authors would like to thank Professor Kenta Ozeki for the help he gave to us during the preparation of this paper.
This work was supported by JSPS KAKENHI Grant number JP23K03204 (to M.F), JP21K03345 (to R.M), 25K07107 (to S.T), 
 and research grant of Senshu University (to S.T).

\end{document}